\pdfoutput=1
\documentclass[a4paper]{amsart}
\usepackage[utf8]{inputenc}
\usepackage{dutchcal}
\usepackage{hyperref}
\usepackage{microtype}

\usepackage{tikz-cd}

\newtheorem{thm}{Theorem}[section]
\newtheorem{prop}[thm]{Proposition}
\newtheorem{lemma}[thm]{Lemma}
\newtheorem{cor}[thm]{Corollary}

\theoremstyle{definition}
\newtheorem{defn}[thm]{Definition}
\newtheorem*{setup}{Setup}

\newtheorem{remark}[thm]{Remark}
\newtheorem{exmpl}[thm]{Example}

\newcommand{\OX}{\ensuremath{\mathcal O_X}}
\newcommand{\OU}{\ensuremath{\mathcal O_U}}
\newcommand{\OXx}{\ensuremath{\mathcal O_{X,x}}}
\DeclareMathOperator{\QCoh}{QCoh}

\newcommand{\QCohX}{\ensuremath{\QCoh(X)}}
\newcommand{\QCohU}{\ensuremath{\QCoh(U)}}
\newcommand{\OXMod}{\ensuremath{\OX\text{\rm-Mod}}}
\newcommand{\OUMod}{\ensuremath{\OU\text{\rm-Mod}}}
\newcommand{\OXxMod}{\ensuremath{\OXx\text{\rm-Mod}}}
\newcommand{\QCohPk}{\ensuremath{\QCoh(\Pk)}}

\newcommand{\qc}{\mathrm{qc}}
\newcommand{\Zg}{\mathrm{Zg}}

\newcommand{\Pk}{\mathbb{P}_k^1}

\newcommand{\D}{\mathbf D}
\newcommand{\Z}{\mathbb Z}
\newcommand{\Zp}{\mathbb Z_{(p)}}
\newcommand{\Zpb}{\overline{\Zp}}
\newcommand{\Q}{\mathbb Q}
\newcommand{\Qpb}{\overline{\mathbb Q_{(p)}}}
\newcommand{\C}{\mathbf C}
\newcommand{\ito}{\hookrightarrow}

\DeclareMathOperator{\Hom}{Hom}
\DeclareMathOperator{\Ext}{Ext}
\DeclareMathOperator{\RHom}{\mathbf{R}Hom}
\DeclareMathOperator{\Hh}{H}
\DeclareMathOperator{\Spec}{Spec}
\newcommand{\HOM}{\operatorname{\mathbcal{Hom}}}

\let\mc\mathcal
\let\mbc\mathbcal

\begin{document}
\title{Purity in categories of sheaves}
\author{Mike Prest}
\address{School of Mathematics, Alan Turing Building, University of Manchester, Manchester M13 9PL, UK}
\email{mprest@manchester.ac.uk}
\author{Alexander Sl\'avik}
\address{Department of Algebra, Charles University, Faculty of Mathematics and Physics, Sokolovsk\'a 83, 186\,75 Prague 8, Czech Republic}
\email{slavik.alexander@seznam.cz}
\thanks{The second author's research was supported from the grant GA \v CR 17-23112S of the Czech Science Foundation, from the grant SVV-2017-260456 of the SVV project and from the grant UNCE/SCI/022 of the Charles University Research Centre.}

\begin{abstract}
We consider categorical and geometric purity for sheaves of modules over a scheme satisfying some mild conditions, both for the category of all sheaves and for the category of quasicoherent sheaves.  We investigate the relations between these four purities and compute a number of examples, in particular describing both the geometric and categorical Ziegler spectra for the category of quasicoherent sheaves over the projective line over a field.
\end{abstract}

\maketitle

The pure-exact sequences in the category of modules over a ring can be characterised in many ways, for example as the direct limits of split exact sequences.  Purity was defined first for abelian groups, then extended to categories of modules and, more recently, to many kinds of additive category; and it has found many uses.  It has also turned out to be of central importance in the model theory of modules and the extension of that to other additive categories.  We will refer to this as {\em categorical purity}, reflecting various of the equivalent ways in which it can be defined.

In the context of (quasicoherent) sheaves of modules over a scheme $X$ (which satisfies some fairly mild conditions) there is, as well as the categorical purity, another good notion of purity which is of interest (see, e.g., \cite{EEO} and references therein).  That purity is defined using the local structure, in particular in terms of restricting to a cover by affine opens, so we refer to this as {\em geometric purity}.  On an affine scheme, the two notions of purity coincide but in general they are different.

Here we investigate these purities and their relationships in the category \OXMod\ of all modules and in the category \QCohX\ of quasicoherent modules, with the particular aim of describing the indecomposable pure-injectives and the space---the Ziegler spectrum---that they form.  Note that, for each (suitable) scheme $X$, there will be four Ziegler spectra: one for each of categorical and geometric purity, and for each category \OXMod\ and \QCohX.  We compute a number of examples, in particular both spectra for the category \OXMod\ over a local affine 1-dimensional scheme in Section \ref{section-Zp}, and both spectra for the category \QCohX\ over the projective line over a field in Section \ref{section-ProjLine}.

\begin{setup}
Throughout, $X$ denotes a scheme, \OXMod\ the category of all \OX-modules and \QCohX\ the category of all quasicoherent modules (sheaves) on~$X$. Whenever $U \subseteq X$ is an open set, \OU\ is a shorthand for the restriction $\OX|_U$.
\end{setup}

For the majority of our results, we will assume that the scheme in question has some topological property:

\begin{defn}
A scheme is called \emph{quasiseparated} if the intersection of any two quasicompact open sets is quasicompact.
A scheme is \emph{concentrated} if it is quasicompact and quasiseparated.
\end{defn}

We continue by defining all the purity-related notions used in the paper.

\begin{defn}
\label{g-pure-def}
A short exact sequence in $0 \to \mc A \to \mc B \to \mc C \to 0$ in \OXMod\ is \emph{geometrically pure} (or \emph{g-pure} for short) if it stays exact after applying the sheaf tensor product functor $- \otimes \mc Y$ for all $\mc Y \in \OXMod$; equivalently, \cite[Proposition 3.2]{EEO}, if the sequence of \OXx-modules $0 \to \mc A_x \to \mc B_x \to \mc C_x \to 0$ is pure-exact for each $x \in X$. We define \emph{g-pure monomorphisms} and \emph{g-pure epimorphisms} in the obvious way. An \OX-module $\mc N$ is \emph{g-pure-injective} if the functor $\Hom_{\OXMod}(-,\mc N)$ is exact on g-pure exact sequences of \OX-modules.
\end{defn}

\begin{defn}
Recall that \QCohX\ is a full subcategory of \OXMod. Following \cite{EEO}, we say that a short exact sequence of quasicoherent sheaves is \emph{g-pure} if it is g-pure in the larger category \OXMod.  By \cite[Propositions 3.3 \& 3.4]{EEO}, this notion of g-purity for quasicoherent sheaves is equivalent to purity after restricting either to all open affine subsets, or to a chosen open affine cover of $X$. We say that a quasicoherent sheaf $\mc N$ is \emph{g-pure-injective in} \QCohX\ if the functor $\Hom_{\QCohX}(-,\mc N)$ is exact on g-pure exact sequences of quasicoherent sheaves.
\end{defn}

\begin{remark}
The notion of g-purity in \QCohX\ could also be established via the property that tensoring with any \emph{quasicoherent} sheaf preserves exactness. By \cite[Remark 3.5]{EEO}, this would give the same for quasiseparated schemes.
\end{remark}

\begin{defn}
If $X$ is a quasiseparated scheme, it(s underlying topological space) has a basis of quasicompact open sets closed under intersections. Therefore, by \cite[3.5]{PR}, the category \OXMod\ is locally finitely presented and as such has a notion of purity: This is defined by exactness of the functors $\Hom_{\OXMod}(\mc F, -)$, where $\mc F$ runs over all finitely presented objects (i.e.\ $\Hom_{\OXMod}(\mc F, -)$ commutes with direct limits). We will call this notion \emph{categorical purity} or \emph{c-purity} for short, defining c-pure-injectivity etc.\ in a similar fashion as in Definition \ref{g-pure-def}.
\end{defn}

\begin{defn}
If $X$ is a concentrated scheme, then the category \QCohX\ is locally finitely presented by \cite[Proposition 7]{G}. Again, all the c-pure notions are defined for \QCohX\ in a natural way.
\end{defn}

We will use \cite{P} as a convenient reference for many of the results that we use concerning purity and definability.

\begin{remark}
If $X$ is a concentrated scheme, then by \cite[Proposition 3.9]{EEO}, c-pure-exact sequences in \QCohX\ are g-pure-exact, and it is easy to see that the proof carries mutatis mutandis to the category \OXMod\ for $X$ quasiseparated. Therefore, in these cases, g-pure-injectivity is a stronger notion than c-pure-injectivity.
\end{remark}

Let us point out that for having a well-behaved (categorical) purity, one does not need a locally finitely presented category; a \emph{definable} category (in the sense of \cite[Part III]{P}) would be enough. However, we are not aware of any scheme $X$ for which \QCohX\ or \OXMod\ would be definable, but not locally finitely presented. Furthermore, the question of when exactly are these categories locally finitely presented does not seem to be fully answered. It is also an open question whether every definable Grothendieck category is locally finitely presented (there is a gap in the argument for this at \cite[3.6]{P5}).

In any case, the situation in this paper is as follows: Section \ref{section-OXMod}, dealing with the category \OXMod, does not need any special assumptions on the scheme $X$ for most of its propositions, therefore it starts with the (slightly obscure) assumption on mere definability of \OXMod. On the other hand, Section \ref{section-QCohX} really needs $X$ to be concentrated almost all the time, a fact that is stressed in all the assertions.

\medskip

If $X$ is an affine scheme, then the category \QCohX\ is equivalent to the category of modules over the ring of global sections of $X$, and both g-purity and c-purity translate to the usual purity in module categories. A converse to this for $X$ concentrated is Proposition \ref{affinity-criterion}. However, as the Section \ref{section-Zp} shows, even for very simple affine schemes, the purities do not coincide in the category \OXMod.

\section{Relation between purity in \texorpdfstring{\OXMod}{O\_X-Mod} and \texorpdfstring{\QCohX}{QCoh(X)}}

Recall that the (fully faithful) forgetful functor $\QCohX \to \OXMod$ has a right adjoint $\C\colon \OXMod \to \QCohX$, usually called the \emph{coherator}. If we need to specify the scheme $X$, we use notation like $\C_X$.

Since g-purity in \QCohX\ is just ``restricted'' g-purity from \OXMod, each quasicoherent sheaf which is g-pure-injective as an \OX-module is also g-pure-injective as a quasicoherent sheaf. The example at the end of Section \ref{section-Zp} shows that even in a quite simple situation, the converse is not true: A g-pure-injective quasicoherent sheaf need not be g-pure-injective in \OXMod.

The following was observed in \cite{EEO}:
\begin{lemma}[{\cite[Lemma 4.7]{EEO}}]
Let $\mc N$ be a g-pure-injective \OX-module. Then its coherator $\C(\mc N)$ is a g-pure-injective quasicoherent sheaf.
\end{lemma}

It is not clear whether the coherator preserves g-pure-exact sequences or at least g-pure monomorphisms. A partial result in this direction is Lemma \ref{factor-through-coherator}.

\medskip

The relation between c-purity in \OXMod\ and \QCohX\ is in general not so clear as for g-purity. Note that while \QCohX\ is closed under direct limits in \OXMod\ (indeed, arbitrary colimits), it is usually not closed under direct products and hence it is not a definable subcategory. However, in the case of $X$ concentrated, much more can be said. We start with an important observation.

\begin{defn}[{\cite[Part III]{P}}]
A functor between definable categories is called \emph{definable} if it commutes with products and direct limits.
\end{defn}

\begin{lemma}
Let $X$ be a concentrated scheme. Then the coherator functor is definable.
\end{lemma}
\begin{proof}
As a right adjoint, $\C$ always commutes with limits and in particular products. By \cite[Lemma B.15]{TT}, it also commutes with direct limits for $X$ concentrated.
\end{proof}

Recall that by \cite[Corollary 18.2.5]{P}, definable functors preserve pure-exactness and pure-injectivity. Hence we obtain the following properties for concentrated schemes:

\begin{lemma}\label{concentrated-properties}
Let $X$ be a concentrated scheme.
\begin{enumerate}
\item The coherator functor preserves c-pure-exact sequences and c-pure-injec\-tivity.
\item A short exact sequence of quasicoherent sheaves is c-pure-exact in \QCohX\ if and only if it is also c-pure-exact in the larger category \OXMod.
\item A quasicoherent sheaf c-pure-injective in \OXMod\ is c-pure-injective in \QCohX.
\item A quasicoherent sheaf finitely presented in \QCohX\ is finitely presented in \OXMod. 
\end{enumerate}
\end{lemma}
\begin{proof}
(1) follows from definability of coherator.

(2) Since \QCohX\ is locally finitely presented, every c-pure-exact sequence is the direct limit of split short exact sequences. However, direct limits in \QCohX\ are the same as in \OXMod, so we get the ``only if'' part. To see the ``if'' part, note that the coherator acts as the identity when restricted to \QCohX, so the statement follows from (1).

(3) This is a consequence of (2), or we can again argue that the coherator is the identity on \QCohX\ and use (1).

(4) Let $\mc F$ be a finitely presented object of \QCohX, $I$ be a directed set and $(\mc M_i)_{i \in I}$ a directed system in \OXMod. Then
\begin{align*}
\Hom_{\OXMod}\bigl(\mc F, \varinjlim\nolimits_{i \in I} \mc M_i\bigr)
&\cong \Hom_{\QCohX}\bigl(\mc F, \C(\varinjlim\nolimits_{i \in I} \mc M_i)\bigr) \\
&\cong \Hom_{\QCohX}\bigl(\mc F, \varinjlim\nolimits_{i \in I} \C(\mc M_i)\bigr) \\
&\cong \varinjlim\nolimits_{i \in I} \Hom_{\QCohX}\bigl(\mc F, \C(\mc M_i)\bigr) \\
&\cong \varinjlim\nolimits_{i \in I} \Hom_{\OXMod}(\mc F, \mc M_i),
\end{align*}
where the natural isomorphisms are due to (in this order) $\C$ being a right adjoint, $\C$ commuting with direct limits for concentrated schemes, $\mc F$ being finitely presented in \QCohX, and finally the adjointness again.
\end{proof}

Note, however, that the examples at the end of Section \ref{section-QCohX} show that the pure-injectives in \QCohX\ have little in common with the pure-injectives of \OXMod.

\section{Purity in \texorpdfstring{\OXMod}{O\_X-Mod}}
\label{section-OXMod}

\begin{setup}
If, in this section, any assertion involves c-purity or definability, then it is assumed that the scheme $X$ is such that \OXMod\ is a definable category (e.g.\ $X$ is quasiseparated); similarly for \OUMod\ if $U$ is involved, too.
\end{setup}

We start with a more general lemma, which is of its own interest. Recall that a full subcategory $\mbc A$ of a category $\mbc B$ is called \emph{reflective} provided that the inclusion functor $\mbc A \ito \mbc B$ has a left adjoint (usually called the \emph{reflector}).

\begin{lemma}\label{reflection}
Let $\mbc B$ be an additive category with arbitrary direct sums and products and $\mbc A$ a reflective subcategory. Then an object of $\mbc A$ is pure-injective if and only if it is pure-injective as an object of $\mbc B$.
\end{lemma}
\begin{proof}
Recall the Jensen-Lenzing criterion for pure-injectivity (see \cite[4.3.6]{P}): An object $M$ is pure-injective if and only if for any set $I$, the summation map $M^{(I)} \to M$ factors through the natural map $M^{(I)} \to M^I$. Also recall that while limits of diagrams in $\mbc A$ coincide with those in $\mbc B$, colimits in $\mbc A$ are computed via applying the reflector to the colimit in $\mbc B$.

Let $M$ be an object of $\mbc A$, $I$ any set and denote by $M^{(I)}$ the coproduct in $\mbc A$ and $M^{(I)_{\mbc B}}$ the coproduct in $\mbc B$. By the adjunction, maps from $M^{(I)}$ to objects of $\mbc A$ naturally correspond to maps from $M^{(I)_{\mbc B}}$ to objects of $\mbc A$, hence the summation map $M^{(I)} \to M$ factors through $M^{(I)} \to M^I$ if and only if $M^{(I)_{\mbc B}} \to M$ factors through $M^{(I)_{\mbc B}} \to M^I$: the naturality of adjunction ensures that the factorizing map fits into the commutative diagram regardless of which direct sum we pick.
\end{proof}

\begin{cor}
A sheaf of \OX-modules is c-pure-injective regardless if it is viewed as a sheaf or a presheaf.
\end{cor}
\begin{proof}
The category of sheaves is always a reflective subcategory of presheaves, sheafification being the reflector.
\end{proof}

\begin{cor}\label{open-pinj}
Let $\mc N$ be a c-pure-injective \OX-module and $U \subseteq X$ open. Then $\mc N(U)$ is a pure-injective $\OX(U)$-module.
\end{cor}
\begin{proof}
Let $I$ be any set. By the previous corollary, $\mc N$ is also a (c-)pure-injective object in the category of presheaves, so the summation map from the presheaf coproduct $\mc N^{(I)\,\mathrm{pre}} \to \mc N$ factors through $\mc N^{(I)\,\mathrm{pre}} \ito \mc N^I$. However, for presheaves we have $\mc N^I(U) \cong \mc N(U)^I$ and (unlike for sheaves) $\mc N^{(I)\,\mathrm{pre}}(U) \cong \mc N(U)^{(I)}$, so we have the desired factorization for $\mc N(U)$ as well.
\end{proof}

In the case of concentrated open sets we can say even more:

\begin{lemma}\label{sections-definable}
Let $U \subseteq X$ be a concentrated open set. Then the functor of sections over $U$, $\mc M \mapsto \mc M(U)$, is definable.
\end{lemma}
\begin{proof}
The functor of sections commutes with products for any open set. Furthermore, if the open set is concentrated, then this functor also commutes with direct limits by \cite[009E]{Stacks}.
\end{proof}

Since g-purity is checked stalk-wise, it is useful to overview the related properties of skyscrapers; recall that if $x \in X$ and $M$ is an \OXx-module, then the \emph{skyscraper (sheaf)} $\iota_{x,*}(M)$ is an \OX-module given by
\[ \iota_{x,*}(M)(U) = \begin{cases}
M & \text{if }x \in U, \\
0 & \text{otherwise}
\end{cases} \]
for $U \subseteq X$ open. This gives a functor $\OXxMod \to \OXMod$, which is a fully faithful right adjoint to taking stalks at $x$.

\begin{lemma}\label{skyscraper-definable}
For every $x \in X$ the functor $\iota_{x,*}$ is definable.
\end{lemma}
\begin{proof}
Commuting with products is clear. If $(M_i)_{i \in I}$ is a directed system of \OXx-modules, then $\iota_{x,*}(\varinjlim_{i \in I} M_i)$ coincides with the ``section-wise direct limit'', i.e.\ direct limit in the presheaf category. Since this direct limit is again a skyscraper, hence a sheaf, it is the direct limit also in the sheaf category.
\end{proof}

From Lemma \ref{skyscraper-definable} we know that skyscrapers built from pure-injective modules are c-pure-injective. It is easy to see that they are even g-pure-injective:

\begin{lemma}\label{pinj-sky}
Let $x \in X$ and $N$ be a pure-injective \OXx-module. Then $\iota_{x,*}(N)$ is a g-pure-injective \OX-module.
\end{lemma}
\begin{proof}
Let $\mc A \ito \mc B$ be a g-pure monomorphism in \OXMod. Using the adjunction, checking that
\[ \Hom_{\OXMod}(\mc B, \iota_{x,*}(N)) \to \Hom_{\OXMod}(\mc A, \iota_{x,*}(N)) \]
is surjective is equivalent to checking that
\[ \Hom_{\OXxMod}(\mc B_x, N) \to \Hom_{\OXxMod}(\mc A_x, N) \]
is surjective. But this is true since $\mc A_x \ito \mc B_x$ is a pure mono in \OXxMod.
\end{proof}

The preceding observation allows us to establish a property of g-purity similar to that of c-purity:

\begin{cor}\label{g-purity-check}
A short exact sequence in \OXMod\ is g-pure-exact if and only if it remains exact after applying the functor $\Hom_{\OXMod}(-,\mc N)$ for every g-pure-injective \OX-module $\mc N$.
\end{cor}
\begin{proof}
The ``only if'' part is clear. The ``if'' follows from the fact that purity in \OXxMod\ (where $x \in X$) can be checked using the functors $\Hom_{\OXxMod}(-,N)$, where $N$ is pure-injective, the adjunction between stalk and skyscraper, and Lemma \ref{pinj-sky}.
\end{proof}

We recall the very general argument for the fact referred to in the above proof, since we use it elsewhere. If $f\colon A\to B$ is an embedding and the notion of purity is such that there is a pure embedding $h\colon A\to N$ for some pure-injective $N$, and if this factors through $g\colon B\to N$, then, using $gf$ pure implies $f$ pure, we deduce purity of $f$.

We proceed with investigating what purity-related notions are preserved under various functors between sheaf categories. Let $U \subseteq X$ be open; then there are the following three functors:
\begin{enumerate}
\item $\iota_{U,!}\colon \OUMod\to\OXMod$, the \emph{extension by zero},
\item $(-)|_U\colon \OXMod\to\OUMod$, the \emph{restriction},
\item $\iota_{U,*}\colon \OUMod\to\OXMod$, the \emph{direct image} (also called the \emph{pushforward}).
\end{enumerate}
These three form an adjoint triple
\[ \iota_{U,!} \dashv (-)|_U \dashv \iota_{U,*} \]
with the outer two functors fully faithful---composing any of them with the restriction gives the identity. Consequently, all the three functors preserve stalks at any point of $U$. Finally, $\iota_{U,!}$ is always exact; see e.g.\ \cite[\S\S II.4, II.6]{I} for details.

\begin{lemma}\label{restriction-definable}
Let $U \subseteq X$ be an open set. Then the restriction functor $(-)|_U$ is definable and preserves g-pure-exactness and g-pure-injectivity.
\end{lemma}
\begin{proof}
Definability follows from the fact that this functor has both a right and a left adjoint. Since restriction preserves stalks, it preserves g-pure-exactness. Finally, let $\mc N \in \OXMod$ be g-pure-injective and $\mc A \ito \mc B$ a g-pure monomorphism of \OU-modules. Then the monomorphism $\iota_{U,!}(\mc A) \ito \iota_{U,!}(\mc B)$ is g-pure, because on $U$ the stalks remain the same, whereas outside $U$ they are zero (see, e.g., \cite[p.\ 106]{I}). Straightforward use of the adjunction then implies that $\mc N|_U$ is g-pure-injective in \OUMod.
\end{proof}


%

Note that extension by zero does not commute with products in general, as the following example shows, therefore there is no hope for definability. It also preserves neither c- nor g-pure-injectivity:

\begin{exmpl}\label{ext-by-zero-nondef}
Let $p \in \Z$ be a prime number, $X = \Spec(\Z)$ and $U = \Spec(\Z[p^{-1}]) \subseteq X$. For a prime number $q \neq p$, let $\mc N^q = \iota_{(q),*}(\Z[p^{-1}]/(q))$, i.e.\ the skyscraper coming from the $q$-element group, regarded as an \OU-module. By Lemma \ref{pinj-sky}, this is a g-pure-injective \OU-module, and so is $\mc N = \prod_{q \neq p}\mc N^q$. Let $\mc M = \iota_{U,!}(\mc N)$; we are going to show that the global sections of $\mc M$ are not a pure-injective $\Z$-module, thus (Corollary \ref{open-pinj}) $\mc M$ is not even c-pure-injective.

Since extension by zero preserves all sections within $U$, $\mc M(U)$ 
is the product of all $q$-element groups for $q$ a prime distinct from $p$. By \cite[Definition 6.1]{I}, $\mc M(X)$ consists of those elements of $\mc M(U)$, whose support in $U$ is closed in $X$. Now every element of $\mc M(U)$ is either torsion, in which case its support is a finite union of closed points, therefore closed in $X$, or torsion-free, which has a non-zero stalk at the generic point and hence on each point of $U$, but $U$ is not closed in $X$. We infer that $\mc M(X)$ is the torsion part of $\mc M(U)$, in other words the direct sum inside the direct product. However, this is not pure-injective, as it is torsion and reduced, but not bounded.

Finally, the global sections of the sheaf $\iota_{U,!}(\mc N^q)$ form the $q$-element group, hence we see that the global sections of the product $\mc P = \prod_{q \neq p}\iota_{U,!}(\mc N^q)$ are different from $\mc M(X)$ (namely, $\mc P(X)$ is the product of $q$-element groups over all primes $q\neq p$), hence $\iota_{U,!}$ does not commute with products.
\end{exmpl}

Here is what we can actually prove:

\begin{lemma}
Let $U \subseteq X$ be an open set. Then the functor $\iota_{U,!}$ preserves and reflects c- and g-pure-exact sequences, and reflects c- and g-pure-injectivity.
\end{lemma}
\begin{proof}
A monomorphism $\mc A \ito \mc B$ in \OXMod\ is c-pure (g-pure) if and only if the contravariant functor $\Hom_{\OXMod}(-,\mc N)$ turns it into a surjection for every c-pure-injective (g-pure-injective) \OX-module $\mc N$ (for g-purity, this is Corollary \ref{g-purity-check}); similarly for monomorphisms in \OUMod.

Due to the adjunction, the surjectivity of
\[ \Hom_{\OXMod}(\iota_{U,!}(\mc B), \mc N) \to \Hom_{\OXMod}(\iota_{U,!}(\mc A), \mc N) \]
is equivalent to the surjectivity of
\[ \Hom_{\OUMod}(\mc B, \mc N|_U) \to \Hom_{\OUMod}(\mc A, \mc N|_U) \]
and by Lemma \ref{restriction-definable}, $\mc N|_U$ is c-pure-injective (g-pure-injective) if $\mc N$ is, so the preservation of pure-exactness follows.

If a short exact sequence in \OUMod\ becomes c-pure-exact (g-pure-exact) after applying $\iota_{U,!}$, then the original sequence has to be c-pure-exact (g-pure-exact), too, since it can be recovered by applying the restriction functor, which preserves both types of pure-exactness by Lemma \ref{restriction-definable}. The same argument shows that $\iota_{U,!}$ reflects pure-injectivities as well.
\end{proof}

Finally, we investigate the properties of direct image.

\begin{lemma}
Let $U \subseteq X$ be an open set. Then the functor $\iota_{U,*}$ preserves and reflects c- and g-pure-injectivity.
\end{lemma}
\begin{proof}
That $\iota_{U,*}$ reflects c-pure-injectivity (g-pure-injectivity) is clear because its composition with restriction to $U$ produces the identity functor. If $\mc A \ito \mc B$ is c-pure (g-pure) in \OXMod, then so is $\mc A|_U \ito \mc B|_U$, hence
\[ \Hom_{\OUMod}(\mc B|_U, \mc N) \to \Hom_{\OUMod}(\mc A|_U, \mc N) \]
is surjective for $\mc N \in \OUMod$ c-pure-injective (g-pure-injective). The adjunction implies the surjectivity of
\[ \Hom_{\OXMod}(\mc B, \iota_{U,*}(\mc N)) \to \Hom_{\OXMod}(\mc A, \iota_{U,*}(\mc N)) \]
and we conclude that $\iota_{U,*}(\mc N)$ is c-pure-injective (g-pure-injective).

Alternatively, for the c-pure-injectivity, we may argue as follows: Note that the full faithfulness of $\iota_{U,*}$ allows us to view $\OUMod$ as a reflective subcategory of $\OXMod$, restriction functor being the reflector. The statement about c-pure-injectivity thus follows from Lemma \ref{reflection}.
\end{proof}

For special open sets $U$ we obtain a definable functor:

\begin{lemma}\label{direct-image-definable}
Let $U \subseteq X$ be an open set such that $\iota_U\colon U\ito X$ is a concentrated morphism (i.e.\ for every $V \subseteq X$ affine open, the intersection $U \cap V$ is concentrated). Then the functor $\iota_{U,*}$ is definable. If $X$ is concentrated, then the essential image of $\iota_{U,*}$ is a definable subcategory of \OXMod.
\end{lemma}
\begin{proof}
Commuting with products follows from the fact that $\iota_{U,*}$ is right adjoint. Commuting with direct limits is a special case of \cite[Lemma B.6]{TT}.

If $X$ is concentrated, then $U$ is concentrated as well by \cite[Lemma 16]{M2}, and the intersection of any concentrated open $V \subseteq X$ with $U$ is concentrated, too. Therefore, by Lemma \ref{sections-definable}, for each concentrated $V \subseteq X$, we have the definable functors of sections over $V$ and $U \cap V$. Since direct limits and direct products are exact in categories of modules, we have another two definable functors, $K_V$ and $C_V$, assigning to $\mc M$ the kernel and the cokernel of the restriction $\mc M(V) \to \mc M(U \cap V)$, respectively.

As concentrated open sets form a basis of $X$, we see that $\mc M \in \OXMod$ is in the essential image of $\iota_{U,*}$ if and only if $K_V(\mc M) = C_V(\mc M) = 0$ for every $V \subseteq X$ open concentrated. Therefore the essential image in question, being the intersection of kernels of definable functors, is a definable subcategory.
\end{proof}

\begin{remark}
Let us point out that \cite[Lemma B.6]{TT} has a much wider scope: Indeed, for any concentrated map of schemes $f$, the direct image functor $f_*$ is definable.
\end{remark}

Let us now focus on the weaker notion of geometric purity. We start with observing that there is a plenty of naturally arising g-pure monomorphisms around.

\begin{lemma}\label{UV-pure}
Let $U \subseteq V$ be open subsets of $X$ and denote by $\iota_U$, $\iota_V$ their inclusions into $X$. Then for any $\mc M \in \OXMod$ the natural map
\[ \iota_{U,!}(\mc M|_U) \to \iota_{V,!}(\mc M|_V) \]
is a g-pure monomorphism.
\end{lemma}
\begin{proof}
Passing to stalks at $x \in X$, we see that the map is either the identity (if $x \in U$) or a map from the zero module (otherwise), hence a pure monomorphism of \OXx-modules. Hence the map is even stalkwise split.
\end{proof}

Recall that a sheaf is called \emph{flasque} if all its restriction maps are surjective. Recall also that for \OX-modules $\mc A$, $\mc B$, the \emph{sheaf hom}, which we denote by $\HOM_X(\mc A, \mc B)$, is the \OX-module defined via
\[ \HOM_X(\mc A, \mc B)(U) = \Hom_{\OUMod}(\mc A|_U, \mc B|_U) \]
for every open $U \subseteq X$, with the obvious restriction maps.

\begin{cor}\label{g-pure-flasque}
If $\mc N \in \OXMod$ is g-pure-injective and $\mc A \in \OXMod$ arbitrary, then the sheaf hom \OX-module $\HOM_X(\mc A, \mc N)$ is flasque. In particular, $\mc N \cong \HOM_X(\OX, \mc N)$ is flasque.
\end{cor}
\begin{proof}
Apply the functor $\Hom_{\OXMod}(-,\mc N)$ to the g-pure monomorphism (Lemma \ref{UV-pure}) $\iota_{U,!}(\mc A|_U) \to \iota_{V,!}(\mc A|_V)$ and use the adjunction
\[ \Hom_{\OXMod}(\iota_{U,!}(\mc A|_U),\mc N) \cong \Hom_{\OUMod}(\mc A|_U,\mc N|_U). \]
\end{proof}


\begin{lemma}\label{shriek-star-pure}
Let $U \subseteq X$ be open. Then the natural map
\[ \iota_{U,!}(\mc M|_U) \to \iota_{U,*}(\mc M|_U) \]
is a g-pure monomorphism for every $\mc M \in \OXMod$.
\end{lemma}
\begin{proof}
Exactly the same as for Lemma \ref{UV-pure}---passing to stalks at $x \in X$, we see that the map is either the identity (if $x \in U$) or a map from the zero module (otherwise), hence a pure monomorphism of \OXx-modules.
\end{proof}

\begin{cor}\label{gpinj-split}
Let $\mc N$ be g-pure-injective \OX-module. Every restriction map in $\mc N$ is a split epimorphism. For every $U \subseteq X$ open, $\iota_{U,*}(\mc N|_U)$ is a direct summand of $\mc N$.
\end{cor}
\begin{proof}
Let $\mc N \in \OXMod$ be g-pure-injective and $U$, $V$ be open subsets of $X$. Using Lemma \ref{shriek-star-pure}, we obtain a surjection
\[ \Hom_{\OXMod}(\iota_{U,*}(\mc N|_U), \mc N) \to \Hom_{\OXMod}(\iota_{U,!}(\mc N|_U), \mc N) \cong \Hom_{\OUMod}(\mc N|_U, \mc N|_U). \]
Hence there is a map $f\colon \iota_{U,*}(\mc N|_U) \to \mc N$ which, after restricting to the subsheaf $\iota_{U,!}(\mc N|_U)$, corresponds to the identity map in the adjunction, thus being the identity when restricted to subsets of $U$. Note that
\[ \iota_{U,*}(\mc N|_U)(V) = \mc N(U \cap V); \]
put $W = U \cap V$. Since the action of $f$ on sections on $V$ and $W$ commutes with restriction maps (i.e.\ $f$ is a sheaf homomorphism), we obtain a commutative square
\[ \begin{tikzcd}[row sep=large,column sep=6em,every label/.append style = {font = \normalsize}]
\mc N(W) = \iota_{U,*}(\mc N|_U)(V) \arrow[r,"f^V"] \arrow[d,"\mathrm{res}_{WV}^{\iota_{U,*}(\mc N|_U)} = \mathrm{id}_{\mc N(W)}"'] & \mc N(V) \arrow[d,"\mathrm{res}_{WV}^{\mc N}"] \\
\mc N(W) = \iota_{U,*}(\mc N|_U)(W) \arrow[r,"f^W = \mathrm{id}_{\mc N(W)}"] & \mc N(W)
\end{tikzcd} \]
We see that $\mathrm{res}_{WV}^{\mc N} \circ f^V$ is an epi-mono factorization of the identity map on $\mc N(W)$, from which both assertions in the statement follow---$\mathrm{res}_{WV}^{\mc N}$ is a split epimorphism and $f$ is a split embedding of $\iota_{U,*}(\mc N|_U)$ into $\mc N$.
\end{proof}

\begin{cor}
Let $\mc N$ be an indecomposable g-pure-injective \OX-module. Then there is $x \in X$ and an indecomposable pure-injective \OXx-module $N$ such that $\mc N \cong \iota_{x,*}(N)$.
\end{cor}
\begin{proof}
Let us first show that restriction maps in $\mc N$ are either isomorphisms or maps to the zero module (but not maps \emph{from} a zero module to a non-zero one). Assume this is not the case, so let $U$ be an open subset of $X$ such that $\ker \mathrm{res}_{XU}^{\mc N}$ is a proper non-zero direct summand of $\mc N(X)$. In that case, by Corollary \ref{gpinj-split}, $\iota_{U,*}(\mc N|_U)$ is a non-trivial proper direct summand of $\mc N$, a contradiction.

Secondly, let $S$ be the support of $\mc N$, i.e.\ the set of those points $x \in X$ such that $\mc N_x$ is non-zero; because of the above-described nature of restriction maps in $\mc N$, $S$ coincides with the set of those $x \in X$ whose \emph{every} open neighborhood has non-zero sections. This is a closed subset of $X$, because if $x \in X \setminus S$, then there is an open set $U \subseteq X$ containing $x$ with $\mc N(U) = 0$, hence $\mc N_y = 0$ for all $y \in U$ and $U \subseteq X \setminus S$.

Next we show that $S$ is irreducible. Let $U, V \subseteq X$ be open sets satisfying $U \cap V \cap S = \emptyset$ ($\Leftrightarrow \mc N(U \cap V) = 0$), we want to show that $U \cap S = \emptyset$ ($\Leftrightarrow \mc N(U) = 0$) or $V \cap S = \emptyset$ ($\Leftrightarrow \mc N(V) = 0$). Since $\mc N$ is a sheaf the assumption implies that the map $\mc N(U \cup V) \to \mc N(U) \times \mc N(V)$ is an isomorphism but then, by the first paragraph, only one of $\mc N(U)$, $\mc N(V)$ can be non-zero, as desired.

We conclude that for any open set $U \subseteq X$, $\mc N(U)$ is non-zero if and only if $U$ intersects the irreducible closed set $S$, i.e.\ contains its generic point $x$. Given the description of restriction maps above, we infer that $\mc N$ is indeed a skyscraper based on $x$. The associated \OXx-module has to be pure-injective by Corollary \ref{open-pinj} and clearly has to be indecomposable, too.
\end{proof}

\begin{lemma}
For any $\mc M \in \OXMod$ the natural map
\[ \mc M \to \prod_{x \in X} \iota_{x,*}(\mc M_x) \]
is g-pure-monomorphism.
\end{lemma}
\begin{proof}
Pick $x \in X$. Then $\mc M_x \cong \bigl(\iota_{U_x}(\mc M|_{U_j})\bigr)_x$, hence the projection on the $x$-th coordinate of the product is a splitting and the map is stalkwise split.
\end{proof}

The following has been already observed in \cite{EEO}, where character modules were used to give a proof. We give a different, slightly more constructive proof.

\begin{lemma}\label{embed-into-g-pure-inj}
Every $\mc M \in \OXMod$ embeds g-purely into a g-pure-injective \OX-module. This module can be chosen to be a product of indecomposable g-pure-injectives.
\end{lemma}
\begin{proof}
For each $x \in X$, pick a pure embedding $\mc M_x \ito N_x$, where $N_x$ is a pure-injective \OXx-module; $N_x$ can be chosen to be a product of indecomposables by \cite[Corollary 5.3.53]{P}. This gives rise to a map $\mc M \to \iota_{x,*}(N_x)$, the skyscraper being g-pure-injective by Lemma \ref{pinj-sky}. Taking the diagonal of these maps we obtain a map
\[ \mc M \to \prod_{x \in X} \iota_{x,*}(N_x). \]
To show that this is a g-pure monomorphism, pick $y \in X$ and passing to stalks at $y$ we have
\[ \mc M_y \to N_y \oplus \Bigl(\prod_{\substack{x \in X \\ x \neq y}}\iota_{x,*}(N_x)\Bigr)_y \]
which is a pure monomorphism after projecting on the left-hand direct summand, hence a pure monomorphism.
\end{proof}

\begin{lemma}\label{natural-pure}
Let $(U_i)_{i \in I}$ be an open cover of $X$ and $\mc M \in \OXMod$. Then the natural map
\[ \mc M \to \prod_{i \in I} \iota_{U_i,*}(\mc M|_{U_i}) \]
is a g-pure-monomorphism.
\end{lemma}
\begin{proof}
Pick $x \in X$ and assume that $x \in U_j$, where $j \in I$. Then $\mc M_x \cong \bigl(\iota_{U_j,*}(\mc M|_{U_j})\bigr)_x$, therefore as in the previous lemma the projection on the $j$-th coordinate of the product is a splitting and the map is stalkwise split.
\end{proof}


\section{Example: Spectrum of \texorpdfstring{$\Zp$}{Z\_(p)}}\label{section-Zp}

In this section we investigate the properties of sheaves over the affine scheme $\Spec(\Zp)$, where $p \in \Z$ is any prime number and $\Zp$ denotes the localization of the ring of integers $\Z$ at the prime ideal $(p)$. Below, $X$ denotes $\Spec(\Zp)$.

Since $X$ is affine, the category \QCohX\ is equivalent to the category of modules over the discrete valuation ring $\Zp$. Purity in such a category is well understood, hence we will not focus on it here at all.

As a topological space, $X$ has two points, $(p)$ and $0$. Its non-empty open sets are $Y = \{0\}$ and $X$, with $\OX(Y) = \Q$ and $\OX(X) = \Zp$. It is straightforward that any presheaf of \OX-modules is automatically a sheaf\footnote{Let us point out here that even though ``there is no non-trivial covering of any open set'', the sheaf axiom in general has the extra consequence that sections over the empty set are the final object of the category. Therefore, e.g.\ sheaves of abelian groups over this two-point space form a proper subcategory of presheaves, which need not assign the zero group to the empty set (!). However, since we always assume \OX\ to be a \emph{sheaf} of rings, its ring of sections over the empty set is the zero ring, over which any module is trivial.}; therefore, the objects $\mc M$ of \OXMod\ are described by the following data: a $\Zp$-module $\mc M(X)$, a $\Q$-module (vector space) $\mc M(Y)$, and a $\Zp$-module homomorphism $\mc M(X) \to \mc M(Y)$. This category is also easily seen to be equivalent to the category of right modules over the ring
\[ R = \begin{pmatrix} \Zp & \Q \\ 0 & \Q \end{pmatrix}. \]

This equivalence translates all c-pure notions in \OXMod\ to ordinary purity in $\text{Mod-}R$. As for g-purity, note that in this simple setting, passing to stalks at a point corresponds to passing to the smallest open subset containing the point. Therefore, a short exact sequence of \OX-modules is g-pure-exact if and only if it is pure exact after passing to global sections (g-purity on $Y$ holds always).

The (right) Ziegler spectrum of $R$ was described in \cite[\S 4.1]{R}. Let us give here an overview of the points (where CB denotes Cantor-Bendixson rank in the Ziegler spectrum):
\begin{center}
\begin{tabular}{|c|c|c|c|c|c|}
\hline
$\mc N(X)$ & $\mc N(Y)$ & CB rank & injective & g-pure-inj. & quasicoh. \\\hline
$\Z_{p^\infty}$ & $0$ & 1 & \checkmark & \checkmark & \checkmark \\\hline
$\Q$ & $0$ & 2 & \checkmark & \checkmark &  \\\hline
$\Q$ & $\Q$ & 1 & \checkmark & \checkmark & \checkmark \\\hline
$\Zp/(p^k)$ & $0$ & 0 & & \checkmark & \checkmark \\\hline
$\Zpb$ & $0$ & 1 & & \checkmark & \\\hline
$\Zpb$ & $\Qpb$ & 0 & & & \checkmark \\\hline
$0$ & $\Q$ & 1 & & & \\\hline
\end{tabular}
\end{center}
The map $\mc N(X) \to \mc N(Y)$ is always the obvious one (identity in the first non-trivial case, inclusion in the second). In all the positive cases, g-pure-injectivity follows directly from Lemma \ref{pinj-sky}. On the other hand, the remaining two \OX-modules are not flasque and therefore cannot be g-pure-injective (Corollary \ref{g-pure-flasque}).

Note, however, that the penultimate module \emph{is} g-pure-injective in the category \QCohX, since it corresponds to the pure-injective $\Zp$-module $\Zpb$.

\section{Purity in \texorpdfstring{\QCohX}{QCoh(X)}}
\label{section-QCohX}

While the previous section showed that even for very simple schemes $X$, the two purities differ in the category \OXMod, it is not so difficult to answer when they coincide in \QCohX, at least for concentrated schemes:

\begin{prop}\label{affinity-criterion}
Let $X$ be a concentrated scheme. Then $X$ is affine if and only if c-purity and g-purity in \QCohX\ coincide.
\end{prop}
\begin{proof}
The ``only if'' part is clear. On the other hand, observe that the structure sheaf \OX\ is locally flat (i.e.\ flat on every open affine subset of $X$), hence every short exact sequence of quasicoherent sheaves ending in it is g-pure. If, moreover, the sequence is g-pure, then it splits, since \OX\ is finitely presented for $X$ concentrated. This means that if the two purities coincide, the first cohomology functor $\Hh^1(X,-) = \Ext^1_{\QCohX}(\OX,-)$ vanishes on all quasicoherent sheaves. By Serre's criterion \cite[01XF]{Stacks}, this implies that $X$ is affine.
\end{proof}

The following example shows that some sort of finiteness condition on the scheme is indeed necessary to obtain the result.

\begin{exmpl}
Let $k$ be a field and $X = (\Spec k)^{(\mathbb N)}$ the scheme coproduct of countably many copies of $\Spec k$. As a topological space, $X$ is a countable discrete space. Every sheaf of \OX-modules is quasicoherent and the category \QCohX\ is actually equivalent to the category of countable collections of $k$-vector spaces with no relations at all. Such a category is semisimple, hence all short exact sequences are both c-pure- and g-pure-exact.
\end{exmpl}

We proceed with deeper investigation of geometric purity in \QCohX. The following analogue of Lemma \ref{g-purity-check} and the similar criterion for c-purity will be useful; note that there is no restriction imposed on the scheme.

\begin{lemma}\label{g-purity-check-qc}
A short exact sequence $0 \to \mc A \to \mc B \to \mc C \to 0$ in \QCohX\ is g-pure-exact if and only if for every g-pure-injective $\mc N \in \QCohX$, the sequence
\[ 0 \to \Hom_{\QCohX}(\mc C, \mc N) \to \Hom_{\QCohX}(\mc B, \mc N) \to \Hom_{\QCohX}(\mc A, \mc N) \to 0 \]
is exact.
\end{lemma}
\begin{proof}
The ``only if'' part is clear. To verify the ``if'' part, recall that by Corollary \ref{g-purity-check}, the sequence is g-pure exact (in \OXMod, which is equivalent), if it stays exact after applying $\Hom_{\OXMod}(-, \mc N)$ for every $\mc N$ g-pure-injective in \OXMod. The result now follows by using the coherator adjunction and the fact that coherator preserves g-pure-injectives.
\end{proof}

\begin{lemma}\label{factor-through-coherator}
Let $\mc M \in \QCohX$ be g-purely embedded into $\mc N \in \OXMod$. Then this embedding factors through a g-pure monomorphism $\mc M \ito \C(\mc N)$.
\end{lemma}
\begin{proof}
Since $\C$ is a right adjoint, we have a map $f\colon\mc M \to \C(\mc N)$, through which the original embedding factorizes. Passing to stalks at $x \in X$, we have a factorization  of pure embeddings of \OXx-modules and an appeal to \cite[Lemma 2.1.12]{P} shows that $f_x$ is a pure monomorphism, hence $f$ is a g-pure monomorphism.
\end{proof}

\begin{lemma}\label{qc-pi-indecomposable}
Let $X$ be a quasicompact scheme, $\mc N \in \QCohX$ indecomposable g-pure-injective and $U_1, \dots, U_n$  a finite open affine cover of $X$. Then there is $i \in \{1, \dots, n\}$ such that $\mc N$ is a direct summand of $\C(\iota_{U_i,*}(\mc N|_{U_i}))$.
\end{lemma}
\begin{proof}
By Lemma \ref{natural-pure}, there is a g-pure monomorphism
\[ \mc N \to \bigoplus_{1\leq i\leq n} \iota_{U_i,*}(\mc N|_{U_i}), \]
which, by Lemma \ref{factor-through-coherator}, yields another g-pure monomorphism of quasicoherent sheaves
\[ \mc N \to \bigoplus_{1\leq i\leq n} \C\bigl(\iota_{U_i,*}(\mc N|_{U_i})\bigr). \]
Since $\mc N$ is g-pure-injective, this map splits; moreover, $\mc N$ has local endomorphism ring, therefore $\mc N$ is a direct summand of one of the summands.
\end{proof}

However, we are not able to say much more for the general quasicompact case. Let us therefore restrict our attention further to concentrated schemes. At this point, it is convenient to clarify the role of the direct image functor.

\begin{remark}\label{direct-image-qc-definable}
Recall that for $X$ concentrated, the functor $\iota_{U,*}$ preserves quasicoherence for every $U \subseteq X$ open affine (even open concentrated is enough, \cite[01LC]{Stacks}). Lemma \ref{direct-image-definable} teaches us that under the same assumptions on $X$ and $U$, $\iota_{U,*}$ is a definable functor from \OUMod\ to \OXMod. Since direct limits in the category of quasicoherent sheaves are the same as those in the larger category of all sheaves of modules, there is no need to care about direct limits. However, direct products do not agree, so some caution has to be exercised here.

Fortunately, if we view $\iota_{U,*}$ solely as a functor from \QCohU\ to \QCohX, then this functor \emph{does} commute with direct products, simply for the reason that it is the right adjoint to the restriction functor from \QCohX\ to \QCohU\ (cf.\ the discussion in \cite[B.13]{TT}). Therefore, the restricted functor $\iota_{U,*}^{\qc}\colon \QCohU \to \QCohX$ is definable.

For each $U \subseteq X$ open affine, the fully faithful functor $\iota_{U,*}^{\qc}$ identifies \QCohU\ with a definable subcategory of \QCohX; the closure under c-pure subsheaves follows from Lemma \ref{direct-image-definable} and the fact that by Lemma \ref{concentrated-properties}, c-purity in \QCohX\ is the same as in \OXMod.
\end{remark}

\begin{cor}\label{qcoh-c-indecomposable}
Let $X$ be a concentrated scheme, $\mc N \in \QCohX$ indecomposable g-pure-injective and $U_1, \dots, U_n$  a finite open affine cover of $X$. Then there is $i \in \{1, \dots, n\}$ such that $\mc N \cong \iota_{U_i,*}^{\qc}(\mc N|_{U_i})$.
\end{cor}
\begin{proof}
Building on Lemma \ref{qc-pi-indecomposable}, we have an $i$ such that the adjunction unit $n\colon \mc N \to \iota_{U,*}(\mc N|_U)$ is a split monomorphism. By the discussion above, the essential image of $\iota_{U_i,*}^{\qc}$ is a definable subcategory, hence it contains $\mc N$. However $n$ restricted to $U$ is the identity map, therefore $n$ is actually an isomorphism.
\end{proof}

\begin{thm}\label{c-ziegler-closed}
Let $X$ be concentrated scheme. Then the indecomposable g-pure-injective quasicoherent sheaves form a closed quasicompact subset of $\Zg(\QCohX)$.
\end{thm}
\begin{proof}
Let $U_1, \dots, U_n$ be a finite open affine cover of $X$. By Corollary \ref{qcoh-c-indecomposable}, every indecomposable g-pure-injective quasicoherent sheaf is of the form $\iota_{U_i,*}^{\qc}(\mc M)$ for some $i \in  \{1, \dots, n\}$  and $\mc M \in \QCoh(U_i)$. We infer that the geometric part of $\Zg(\QCohX)$ is the union of finitely many closed sets corresponding to the ($\iota_{U_i,*}^{\qc}$-images of) definable categories $\QCoh(U_i)$, hence a closed set. Furthermore, since the $U_i$ are affine, we have equivalences $\QCoh(U_i) \cong \OX(U_i)\text{\rm-Mod}$ leading to homeomorphisms $\Zg(\QCoh(U_i)) \cong \Zg(\OX(U_i))$. Therefore, by \cite[Corollary 5.1.23]{P}, the sets in the union are all quasicompact and their (finite) union as well.
\end{proof}

\begin{defn}
Let $X$ be a concentrated scheme. We denote by $\mbc D_X$ the definable subcategory of \QCohX\ corresponding to the closed subset of $\Zg(\QCohX)$ from Theorem \ref{c-ziegler-closed}.
\end{defn}

\begin{remark}\label{direct-image-in-D-X}
Note that by the proof of Theorem \ref{c-ziegler-closed}, for every open affine $U \subseteq X$ and every $\mc M \in \QCohU$, the direct image $\iota^{\qc}_{U,*}(\mc M)$ belongs to $\mbc D_X$.
\end{remark}

\begin{prop}\label{D-X-properties}
Let $X$ be a concentrated scheme.
\begin{enumerate}
\item A c-pure-injective quasicoherent sheaf is g-pure-injective if and only if it belongs to $\mbc D_X$.
\item Any g-pure monomorphism (g-pure-exact sequence) starting in an object of $\mbc D_X$ is c-pure.
\end{enumerate}
\end{prop}
\begin{proof}
It is a standard fact (see \cite[Corollary 5.3.52]{P}) about definable subcategories that their objects are precisely pure subobjects of products of indecomposable pure-injectives. Therefore, if $\mc N \in \mbc D_X$ is c-pure-injective, then it is a direct summand of a product of indecomposable c-pure-injective objects in $\mbc D_X$, all of which are g-pure-injective, a property passing both to products and direct summands, hence $\mc N$ is g-pure-injective.

On the other hand, if $\mc N$ is g-pure-injective and $U_1, \dots, U_n$ a finite open affine cover of $X$, then the g-pure monomorphism
\[ \mc N \to \bigoplus_{1\leq i\leq n} \iota_{U_i,*}^{\qc}(\mc N|_{U_i}) \]
splits. Since $\iota_{U_i,*}^{\qc}(\mc N|_{U_i}) \in \mbc D_X$ for each $1 \leq i \leq n$, we infer that $\mc N \in \mbc D_X$.

For the second claim, let $f \colon \mc M \ito \mc A$ be a g-pure monomorphism with $\mc M \in \mbc D_X$. Assume first $\mc A$ is g-pure-injective; then $\mc A \in \mbc D_X$ by the first part. Recall that in any definable category, purity of a monomorphism can be tested by applying the contravariant Hom functor with every pure-injective object. Since the pure-injectives of $\mbc D_X$ are precisely g-pure-injectives, $f$ ``passes'' this test and is therefore c-pure. If $\mc A \in \QCohX$ is arbitrary, pick a g-pure embedding $g \colon \mc A \ito \mc N$ with $\mc N$ g-pure-injective; such an embedding exists by \cite[Corollary 4.8]{EEO} (or combining Lemmas \ref{embed-into-g-pure-inj} and \ref{factor-through-coherator}). Then $gf$ is a g-pure monomorphism and by the argument above, it is c-pure. We infer that $f$ is c-pure as well.
\end{proof}


\begin{cor}
A concentrated scheme $X$ is affine if and only if the subcategory $\mbc D_X$ is the whole category \QCohX.
\end{cor}
\begin{proof}
If $X$ is affine, then the two purities coincide, therefore c-pure-injectives are g-pure-injectives. On the other hand, if every c-pure-injective is g-pure-injective then, since c-pure-injectives determine which short exact sequences are c-pure-exact, we see that all g-pure-exact sequences are c-pure-exact hence, by Proposition \ref{affinity-criterion}, $X$ is affine.
\end{proof}

\begin{cor}
Let $X$ be a concentrated scheme. Then every indecomposable g-pure-injective $\mc N \in \QCohX$ is the coherator of some indecomposable g-pure-injective $\mc M \in \OXMod$.
\end{cor}
\begin{proof}
By Corollary \ref{qcoh-c-indecomposable}, there is an open affine $U \subseteq X$ such that $\mc N$ is the direct image of an indecomposable g-pure-injective object of \QCohU, which further corresponds to an indecomposable pure-injective $\OX(U)$-module $N$. By \cite[Theorem 2.$\Z$8]{P2}, $N$ is in fact a module over the localization in some maximal ideal of $\OX(U)$; let this maximal ideal correspond to a point $x \in U$. Then clearly $\mc N|_U = \C_U(\iota_{x,*}(N)|_U)$ and by Lemmas \ref{pinj-sky} and \ref{restriction-definable}, $\mc M = \iota_{x,*}(N)$ and $\mc M|_U$ are g-pure-injective (and clearly indecomposable).

By \cite[B.13]{TT}, coherator commutes with direct images of concentrated maps, hence we have
\[ \mc N = \iota_{U,*}^{\qc}\bigl(\C_U(\mc M|_U)\bigr) = \C_X\bigl(\iota_{U,*}(\mc M|_U)\bigr) = \C_X(\mc M), \]
where the last equality follows from the fact that $\mc M$ is a skyscraper.
\end{proof}

Note, however, that the preimage $\mc M$ from the preceding Corollary is far from being unique; Section \ref{section-Zp} gives a couple of examples of indecomposable c-pure-injective \OX-modules with the same module of global sections (and therefore the same coherator).

\begin{prop}
Let $X$ be a concentrated scheme. The following are equivalent for $\mc M \in \QCohX$:
\begin{enumerate}
\item $\mc M \in \mbc D_X$.
\item $\mc M$ is a c-pure subsheaf of a g-pure-injective quasicoherent sheaf.
\item For every finite open affine cover $U_1, \dots, U_n$ of $X$, the monomorphism
\[ \mc M \to \bigoplus_{1\leq i\leq n} \iota_{U_i,*}^{\qc}(\mc M|_{U_i}) \]
is c-pure.
\item There exists a finite open affine cover $U_1, \dots, U_n$ of $X$ such that the monomorphism
\[ \mc M \to \bigoplus_{1\leq i\leq n} \iota_{U_i,*}^{\qc}(\mc M|_{U_i}) \]
is c-pure.
\end{enumerate}
\end{prop}
\begin{proof}
(1) $\Rightarrow$ (2): The definable subcategory $\mbc D_X$ is closed under taking (c-)pure-injective envelopes of its objects; since the pure-injectives in $\mbc D_X$ are g-pure-injective by Proposition \ref{D-X-properties} (1), this produces a c-pure embedding into a g-pure-injective quasicoherent sheaf.

(2) $\Rightarrow$ (1): $\mbc D_X$ contains all g-pure-injectives and is closed under taking c-pure subsheaves.

(1) $\Rightarrow$ (3): The monomorphism in question is always g-pure; by assumption, its domain belongs to $\mbc D_X$, therefore it is c-pure by Proposition \ref{D-X-properties} (2).

(3) $\Rightarrow$ (4) is clear.

(4) $\Rightarrow$ (1): The codomain of the monomorphism belongs to $\mbc D_X$ (Remark \ref{direct-image-in-D-X}), which is closed under taking c-pure subsheaves.
\end{proof}

The following lemma, ensuring that the sheaf cohomology vanishes on $\mbc D_X$, will prove useful in Section \ref{section-ProjLine}.

\begin{lemma}\label{D-cohomology-vanishes}
Let $X$ be a concentrated scheme. Then $\Hh^1(X, \mc M) = 0$ for every $\mc M \in \mbc D_X$.
\end{lemma}
\begin{proof}
Observe that the subcategory of \QCohX, where the first cohomology vanishes, is definable: The closure under direct limits follows from \cite[Lemma B.6]{TT}. The closure under direct products follows from \cite[Corollary A.2]{CS} and the fact that $\Hh^1(X,-) = \Ext_{\QCohX}^1(\OX,-)$. Finally, let $0 \to \mc A \to \mc B \to \mc C \to 0$ be a c-pure-exact sequence in \QCohX, where $\Hh^1(X,\mc B) = 0$. Using c-pure-exactness and Lemma \ref{global-sections-definable}, this sequence stays exact after applying the global sections functor, therefore the map
$\Hh^1(X, \mc A) \to \Hh^1(X, \mc B)$
is injective. Since its codomain vanishes, the same holds for its domain.

Thanks to this observation, it suffices to prove the assertion only for the pure-injectives of $\mbc D_X$, i.e.\ g-pure-injectives. To show that $\Hh^1(X, \mc N) = \Ext_{\QCoh}^1(\OX,\mc N) = 0$ for every g-pure-injective, consider a short exact sequence starting in $\mc N$ and ending in \OX; as \OX\ is flat on each open affine set, the sequence is g-pure-exact, and because $\mc N$ is g-pure-injective, the sequence splits as desired.
%
\end{proof}



%

Let us end this section by observing that the pure-injective objects of \QCohX\ are a bit ``mysterious'' from the sheaf point of view. The best we can say is that by the following lemma, their module of global sections is pure-injective over the ring $\OX(X)$, but that can be far from true on the rest of the open sets.

\begin{lemma}\label{global-sections-definable}
Let $X$ be a concentrated scheme. Then the functor of global sections, viewed as a functor from \QCohX\ to $\OX(X)$-modules, is definable.
\end{lemma}
\begin{proof}
The functor in question is naturally isomorphic to the representable functor $\Hom_{\QCohX}(\OX, -)$, which clearly commutes with products. Commuting with direct limits is \cite[0097]{Stacks}.
\end{proof}

Note that this lemma is very similar to Lemma \ref{sections-definable}; however, in the quasicoherent case, the functor of taking sections on an open set usually does not commute with direct products.

\begin{exmpl}
Proposition \ref{structure-sheaf-pi} below shows that the structure sheaf of the projective line is c-pure-injective in \QCohPk. However, on any open affine set (or stalk at any closed point), this sheaf is not pure injective. Since this violates Lemma \ref{open-pinj}, it is not c-pure-injective in \OXMod.
\end{exmpl}

A slightly more sophisticated example exhibits a similar behaviour even for g-pure-injectives in \QCohX:

\begin{exmpl}
Let $k$ be a field, $R = k[\![x, y]\!]$ the ring of power series in two commuting variables and $X = \Spec R$. Since $R$ is a commutative noetherian complete local domain, it is a pure-injective module over itself by \cite[Theorem 11.3]{JL}. Therefore the structure sheaf \OX\ is a g-pure-injective quasicoherent sheaf. However, for every proper distinguished open affine subset $U \subset X$, $\OX(U)$ is the localization of $R$ in a single element, which is a commutative noetherian domain, but not even local, hence not pure-injective by \cite[Theorem 11.3]{JL} again.
\end{exmpl}

\section{Example: Projective line}
\label{section-ProjLine}

This section is devoted to investigating purity in one of the simplest non-affine schemes, the projective line $\Pk$, where $k$ is any field. Our primary aim is to describe $\Zg(\QCohPk)$, but this scheme also provides several important examples. Since $\Pk$ is a noetherian, hence a concentrated, scheme all the results of the previous section apply.

First of all, recall that the finitely presented objects of \QCohPk\ are precisely the coherent sheaves, and each coherent sheaf decomposes uniquely (in the sense of the Krull-Schmidt Theorem) into a direct sum of indecomposable ones. These indecomposables are of two kinds:
\begin{itemize}
\item line bundles, i.e., the structure sheaf and its twists, which we denote $\mc O(n)$ ($n \in \Z$),
\item torsion sheaves, i.e., skyscrapers $\iota_{x,*}(F)$, where $x \in \Pk$ is a closed point and $F$ is a cyclic torsion module over the DVR $\mc O_{\Pk,x}$.
\end{itemize}
In the following, we simply use $\mc O$ for $\mc O(0)$, which is also the same thing as the structure sheaf $\mc O_{\Pk}$.

\begin{exmpl}\label{line-bundles-example}
For every $a, b, c, d \in \Z$ such that $a < b < d$, $a < c < d$ and $a + d = b + c$, there is a short exact sequence
\[ 0 \to \mc O(a) \to \mc O(b) \oplus \mc O(c) \to \mc O(d) \to 0 \]
in \QCohPk, which is non-split, since $\Hom_{\QCohPk}(\mc O(d), \mc O(b) \oplus \mc O(c)) = 0$. This sequence is not c-pure-exact, since it ends in a finitely presented object but does not split. On the other hand, passing to any stalk or any open affine set, we obtain a split short exact sequence of free modules, therefore the sequence of sheaves is g-pure-exact.
\end{exmpl}

For the projective line, we are able to give a better characterization of g-pure-exactness:
\begin{prop}
Let $0 \to \mc A \to \mc B \to \mc C \to 0$ be a short exact sequence in $\QCohPk$. Then the following statements are equivalent:
\begin{enumerate}
\item The sequence is g-pure-exact.
\item For each (indecomposable) torsion coherent sheaf $\mc T$, the sequence
\[ 0 \to \mc T \otimes \mc A \to \mc T \otimes \mc B \to \mc T \otimes \mc C \to 0 \]
is exact.
\item For each (indecomposable) torsion coherent sheaf $\mc T$, the sequence
\[ 0 \to \Hom_{\QCohPk}(\mc T, \mc A) \to \Hom_{\QCohPk}(\mc T, \mc B) \to \Hom_{\QCohPk}(\mc T, \mc C) \to 0 \]
is exact.
\end{enumerate}
\end{prop}
\begin{proof}
It is clear from the description of coherent sheaves over $\Pk$ that we may restrict to the indecomposable coherent sheaves in (2) and (3).

(1) $\Rightarrow$ (2) is clear from the definition of g-pure-exactness. For (2) $\Rightarrow$ (1) it suffices to recall that the tensor product preserves direct limits, every quasicoherent sheaf is the direct limit of coherent ones, and tensoring by a line bundle is always exact.

To prove (1) $\Leftrightarrow$ (3), first recall that for quasicoherent sheaves, g-pure-exactness is equivalent to pure exactness on each open affine. Observe that the support of an indecomposable torsion coherent sheaf $\mc T$ is a single closed point of $\Pk$; let $U$ be any open affine set containing this point. In such a case, not only is $\mc T$ the direct image of its restriction to $U$, but it is also the extension by zero of this restriction. The adjunction now implies that the exactness of the sequence in (3) is equivalent to the exactness of
\begin{multline*}
 0 \to \Hom_{\QCohU}(\mc T|_U, \mc A|_U) \to \Hom_{\QCohU}(\mc T|_U, \mc B|_U) \to \\ \to \Hom_{\QCohU}(\mc T|_U, \mc C|_U) \to 0.
\end{multline*}
Since $U$ is affine and the ring of sections over $U$ is a PID, this new sequence is exact for every indecomposable torsion coherent $\mc T$ if and only if the sequence of $\mc O(U)$-modules
\[ 0 \to \mc A(U) \to \mc B(U) \to \mc C(U) \to 0 \]
is pure-exact---recall that over a PID, purity can be checked just by using finitely generated (= presented) indecomposable torsion modules, which follows from the description of finitely generated modules. Since $\mc T$ runs over all torsion indecomposable coherent sheaves, $U$ runs over all open affine subsets of $\Pk$, and every torsion $\mc O(U)$-module extends to a torsion coherent sheaf, we are done.
\end{proof}

Let us proceed with describing the indecomposable g-pure-injectives. This is easy thanks to Corollary \ref{qcoh-c-indecomposable}, taking into account that $\Pk$ is covered by affine lines. The Ziegler spectrum of a PID (more generally, a Dedekind domain) is described e.g.\ in \cite[5.2.1]{P}. Therefore, for each closed point $x \in \Pk$, there are the following sheaves:
\begin{itemize}
\item all the indecomposable torsion coherent sheaves based at $x$ (an $\mathbb N$-indexed family),
\item the ``Pr\"ufer'' sheaf: $\iota_{x,*}(P)$, where $P$ is the injective envelope of the unique simple \OXx-module,
\item the ``adic'' sheaf: for every $U \subseteq X$ open, the module of sections is either the completion $\overline{\OXx}$ (if $x \in U$), or the fraction field of this completion (if $x \notin U$); in other words, this is the coherator of $\iota_{x,*}(\overline{\OXx})$.
\end{itemize}
Finally, there is the (g-pure-)injective constant sheaf, assigning to each open set the residue field of the generic point of $\Pk$. We will refer to this sheaf as the generic point of $\Zg(\QCohPk)$.

This observation can be informally summarized by saying that the geometric part of $\Zg(\QCohPk)$ is the ``projectivization'' of the Ziegler spectrum of an affine line.

However, by Proposition \ref{affinity-criterion}, there have to be c-pure-injectives which are not g-pure-injective. An example of this phenomenon is the structure sheaf $\mc O$:

\begin{prop}\label{structure-sheaf-pi}
The structure sheaf is $\Sigma$-c-pure-injective, i.e.\ any direct sum of its copies is c-pure-injective. The same holds for all line bundles.
\end{prop}
We give two proofs, one rather direct, illustrating the technique of computing the coherator over concentrated schemes, the other requiring more framework from model theory. Since twisting is an autoequivalence of \QCohPk, we may restrict to $\mc O$ in both proofs.
\begin{proof}[Elementary proof]
We prove that for any set $I$, the inclusion $\mc O^{(I)} \ito \mc O^I$ splits. Let $U$, $V$ be open affine subsets of $\Pk$ such that $\mc O(U) = k[x]$, $\mc O(V) = k[x^{-1}]$, and $\mc O(U \cap V) = k[x,x^{-1}]$. For the direct sum, the computation is easy: $\mc O^{(I)}(U) = k[x]^{(I)}$, $\mc O^{(I)}(V) = k[x^{-1}]^{(I)}$, and $\mc O^{(I)}(U \cap V) = k[x,x^{-1}]^{(I)}$.

To compute the direct product in \QCohPk, we have to compute the coherator of the product in the category of all sheaves. The way to do that is described in \cite[B.14]{TT}:
\begin{align*}
\mc O^I(U) &= \ker\bigl( \mc O(U)^I \oplus (\mc O(V)^I \otimes_k k[x]) \to \mc O(U \cap V)^I \bigr),\\
\mc O^I(V) &= \ker\bigl( (\mc O(U)^I \otimes_k k[x^{-1}]) \oplus \mc O(V)^I \to \mc O(U \cap V)^I \bigr),\\
\mc O^I(U \cap V) &= \ker\bigl( (\mc O(U)^I \otimes_k k[x^{-1}]) \oplus (\mc O(V)^I \otimes_k k[x]) \to \mc O(U \cap V)^I \bigr).
\end{align*}
Therefore $\mc O^I(U)$ is the submodule of $\mc O(U)^I = k[x]^I$ consisting of sequences of polynomials with bounded degree, similarly for $\mc O^I(V)$; $\mc O^I(U \cap V)$ consists of sequences of polynomials in $x$ and $x^{-1}$ with degree bounded both from above and below.

Observe that $\mc O^I(U)$ can be identified with $k^I[x]$, polynomials over the ring $k^I$ of arbitrary sequences, whereas $\mc O^{(I)}(U)$ corresponds to $k^{(I)}[x]$; analogous assertions hold for $V$ and $U \cap V$. The inclusion of $k$-vector spaces $k^{(I)} \ito k^I$ splits, and this splitting naturally lifts to each of $\mc O(U)$, $\mc O(V)$, $\mc O(U \cap V)$, commuting with the restriction maps, hence defining a splitting of the inclusion $\mc O^{(I)} \ito \mc O^I$ as desired.
\end{proof}

\begin{proof}[Model-theoretic proof]
The category \QCohPk\, being locally finitely presented, is equivalent to the category of flat contravariant functors on its subcategory, $\mbc C = \operatorname{coh}(\Pk)$, of finitely presented objects, via the embedding taking a quasicoherent sheaf $\mc M$ to the representable functor $(-,\mc M)$ restricted to $\mbc C$.  That is a definable subcategory of the category $\operatorname{Mod-}\mbc C$ of all contravariant functors from \QCohPk\ to $k$ (see, e.g., \cite[\S 18]{P3}).  We regard such functors as multisorted modules (as in \cite{P4}), with the set of elements of $\mc M \in \QCohPk$ in sort $(\mc F,-)$, for $\mc F \in \mbc C$, being $(\mc F, \mc M)$.  The model theory of modules then applies.  In particular $\mc M$ is $\Sigma$-pure-injective exactly if each sort $(\mc F,\mc M)$ has the descending chain condition on pp-definable subgroups (cf.\ \cite[Theorem 4.4.5]{P}). Since, for each $\mc F \in \mbc C$, $\Hom_{\QCohPk}(\mc F, \mc O)$ is finite-dimensional over $k$, this descending chain condition is satisfied (every pp-definable subgroup is a $k$-subspace).
\end{proof}

In fact, the list is now complete. To prove this, we use Ziegler spectra of derived categories, but first we need the following observations for the Ext functors and cohomology:


\begin{lemma}\label{Ext-vanishing-definable}
For every $n \in \Z$, the class of objects $\mc M \in \QCohPk$ such that
\[ \Ext^1_{\QCohPk}(\mc O(n), \mc M) = 0 \]
is a definable subcategory of\/ \QCohPk, containing all g-pure-injectives.
\end{lemma}
\begin{proof}
Taking into account that twisting is an autoequivalence of \QCohPk, we have natural equivalence
\[ \Ext^1_{\QCohPk}(\mc O(n), \mc M) \cong \Ext^1_{\QCohPk}(\mc O, \mc M \otimes \mc O(-n)) \cong \Hh^1(\Pk, \mc M \otimes \mc O(-n)) \]
and the proof of definability of the vanishing class continues as in the proof of Lemma \ref{D-cohomology-vanishes}. 
%
%
%

The statement about g-pure-injectives suffices to be checked only for the indecomposable ones. Since these are invariant under twists, this follows from Lemma \ref{D-cohomology-vanishes}.
\end{proof}

We are now prepared to prove that we have successfully identified all indecomposable c-pure-injectives.

\begin{thm}\label{Zg-QCoh-Pk}
Every indecomposable c-pure-injective object of\/ \QCohPk\ is either g-pure-injective or a line bundle.
\end{thm}
\begin{proof}
Let $R = k\tilde A_1$ be the path algebra of the Kronecker quiver. Let $\D(R)$ denote the (unbounded) derived category of the category of right $R$-modules and $\D(\Pk)$ the (also unbounded) derived category of \QCohPk. By the results of \cite{B}, the functor
\[ F = \RHom_{\QCohPk}(\mc O \oplus \mc O(1), -)\colon \D(\Pk) \to \D(R) \]
is a triangulated equivalence. Since $R$ is hereditary, by \cite[Theorem 17.3.22]{P}, its Ziegler spectrum is just the union of all shifts of the Ziegler spectrum of $R$, embedded into $\D(R)$ as complexes with cohomology concentrated in degree $0$.

Let $Z$ be the (representative) set containing all twists of the structure sheaf and all indecomposable g-pure-injectives of \QCohPk. The Ziegler spectrum of $\D(\Pk)$ contains all shifts of $Z$; therefore, since $F$ is an equivalence, to show that $Z$ is indeed the Ziegler spectrum of \QCohPk\ it suffices to show that the shifts of $F(Z)$ cover $\Zg(\D(R))$.

The Ziegler spectrum of $R$ is described in \cite[8.1]{P}, so the checking is only a matter of computation. From Lemma \ref{Ext-vanishing-definable} we get that
\[ \Ext^1_{\QCohPk}(\mc O \oplus \mc O(1), \mc N) = 0 \]
for every indecomposable g-pure-injective $\mc N$, therefore $F$ simply sends the torsion, ``Pr\"ufer'', ``adic'' and generic sheaves to the corresponding points of $\Zg(R)$ put into the same cohomological degree in $\D(R)$.

Using Serre duality \cite[Theorem 7.1]{H}, one obtains the same Ext-vanishing for $\mc O(n)$ for all $n \geq 0$, and it is easy to observe that these line bundles are sent to preprojective $R$-modules via $F$. On the other hand, we have
\[ \Hom_{\QCohPk}(\mc O \oplus \mc O(1), \mc O(n)) = 0 \]
for $n < 0$, and using Serre duality again we get that these line bundles are mapped to preinjective $R$-modules, just shifted to the neighbouring cohomological degree in $\D(R)$.

We conclude that each point of $\Zg(\D(R))$ is a shift of an object from $F(Z)$ as desired.
\end{proof}

Having now the complete description of the points of $\Zg(\QCohPk)$, let us investigate the topology. The geometric part (which forms a closed subset by Theorem \ref{c-ziegler-closed}) is easy to handle and basically follows the description of the Ziegler spectrum of a Dedekind domain (cf.\ \cite[Theorem 5.2.3]{P}). The following observation describes how the line bundles sit in the Ziegler topology:

\begin{prop}\label{Pk-topology}
Every line bundle is a closed and isolated point of\/ $\Zg(\QCohPk)$. Any set of line bundles $\mc O(n)$, where $n$ is bounded from above, is closed. Any set of line bundles $\mc O(n)$, where $n$ is not bounded from above, additionally contains all the ``adic'' sheaves and the generic point in its closure.
\end{prop}
\begin{proof}
This could be deduced from the description of the topology of the Ziegler spectrum over the Kronecker algebra $R$ but we can argue directly as follows.  For each $n \in \Z$, $\mc O(n)$ is the only indecomposable c-pure-injective for which both functors
\[ \Hom_{\QCohPk}(\mc O(n+1), -) \quad \text{and} \quad \Ext_{\QCohPk}(\mc O(n-1), -) \]
vanish; the vanishing class of the former is clearly definable, whereas for the latter we use Lemma \ref{Ext-vanishing-definable}. Therefore the single-point set containing $\mc O(n)$ is closed.

On the other hand, $\mc O(n)$ is the only indecomposable c-pure-injective for which neither of the functors
\[ \Hom_{\QCohPk}(\mc O(n), -) \quad \text{and} \quad \Ext_{\QCohPk}(\mc O(n+2), -) \]
vanishes; the former vanishes on all $\mc O(m)$ for $m < n$, whereas the latter vanishes on g-pure-injectives and all $\mc O(m)$ with $m > n$.

We see that line bundles form a discrete subspace of $\Zg(\QCohPk)$, therefore taking the closure of any set of line bundles can possibly add only points from the geometric part.

If a set $S$ of line bundles $\mc O(n)$ has $n$ strictly bounded above by some $m$, then this set is contained in the definable vanishing class of $\Hom_{\QCohPk}(\mc O(m), -)$, which contains no g-pure-injectives. Hence $S$ is closed.

If, on the other hand, a set $S$ of line bundles contains $\mc O(n)$ with arbitrarily large $n \in \Z$, additional points appear in the closure. Denote by $\mbc X$ the smallest definable subcategory of \QCohX\ containing $S$. Firstly, let $U$ be the complement of a single closed point $x \in \Pk$. For every pair $m < n$ such that $\mc O(m), \mc O(n) \in S$, pick a monomorphism $\mc O(m) \to \mc O(n)$ which is an isomorphism on $U$. This way we obtain a chain of monomorphisms, the direct limit of which is the sheaf $\mc M = \iota^{\qc}_{U,*}(\mc O|_U)$. Therefore $\mc M \in \mbc X$.

Since $\mc O(U)$ is a PID, the pure-injective hull $N$ of $\mc O(U)$ in the category of $\mc O(U)$-modules is the direct product of completions of the local rings $\mc O_y$ for each closed $y \in U$; applying the (definable by \ref{direct-image-qc-definable}) direct image functor $\iota^{\qc}_{U,*}$ to the corresponding map in \QCohU\ produces a c-pure-injective hull $\mc M \ito \mc N$ in \QCohPk, therefore $\mc N \in \mbc X$. However, $\mc N$ is just the product of ``adic'' sheaves, which therefore belong to $\mbc X$, too. Since the choice of the point $x$ was arbitrary, we obtain all ``adic'' sheaves in the closure of $S$. Furthermore, any such point has the generic point in its closure.

Finally, observe that $\mc X$ consists only of torsion-free sheaves, i.e.\ the sheaves for which the definable functors $\Hom_{\QCohPk}(\mc T, -)$, $\mc T$ torsion coherent, vanish. All the remaining indecomposable pure-injectives are not torsion-free and therefore cannot be in the closure of $S$.
\end{proof}


In the case of the projective line, we are also able to give an alternative description of the subcategory $\mbc D_{\Pk}$, which we denote here just $\mbc D$ for short; this description was suggested to us by Jan \v S\v tov\'\i\v cek:
\begin{prop}
An object $\mc M$ of\/ \QCohPk\ belongs to $\mbc D$ if and only if for each $n \in \Z$,
\[ \Ext^1_{\QCohPk}(\mc O(n), \mc M) = 0, \]
if and only if for each $n \in \Z$,
\[ \Hom_{\QCohPk}(\mc M, \mc O(n)) = 0. \]
\end{prop}
\begin{proof}
Let $\mbc E$ be the subcategory of \QCohPk\ consisting of those $\mc M$ such that $\Ext^1_{\QCohPk}(\mc O(n), \mc M) = 0$ for all $n \in \Z$. By Lemma \ref{Ext-vanishing-definable}, $\mbc E$ is the intersection of definable subcategories, therefore it is itself definable.

The proof that $\mbc D = \mbc E$ is now just a matter of checking that $\mbc D$ and $\mbc E$ contain the same indecomposable c-pure-injectives. By Lemma \ref{Ext-vanishing-definable}, all indecomposable g-pure-injectives belong to $\mbc E$, while Example \ref{line-bundles-example} shows that no line bundle belongs to $\mbc E$, which is precisely the case for $\mbc D$ as well.

To check the second claim, pick $\mc M \notin \mbc D$ and let $i \colon \mc M \ito \mc N$ be a c-pure embedding into the direct product of indecomposable c-pure-injectives. Taking into account the description of these indecomposables (Theorem \ref{Zg-QCoh-Pk}), some of the terms in the product have to be line bundles, for otherwise we would have $\mc M \in \mbc D$, and for the same reason the composition of $i$ with the projection on at least one of the line bundles has to be non-zero, showing that $\Hom_{\QCohPk}(\mc M, \mc O(n)) \neq 0$ for some $n \in \Z$.

On the other hand, let $\mc M \in \mbc D$. The description of $\mbc D$ as the intersection of Ext-vanishing classes shows that $\mbc D$ is closed under arbitrary factors as $\Ext^2$ vanishes on \QCohPk. The image of any non-zero map $\mc M \to \mc O(n)$ would be a line bundle, too, but no line bundle belongs to $\mbc D$, therefore we condclude that $\Hom_{\QCohPk}(\mc M, \mc O(n)) = 0$ for each $n \in \Z$ in this case.
\end{proof}

Therefore, for the projective line, we get some extra properties of $\mbc D$:

\begin{cor}
The subcategory $\mbc D$ is a torsion class and a right class of a cotorsion pair. In particular, $\mbc D$ is closed under arbitrary colimits, factors, and extensions.
\end{cor}


\end{document}